\documentclass[11pt]{amsart}
\usepackage{hyperref}
\hypersetup{nesting=true,debug=true,naturalnames=true}
\usepackage{graphicx,amssymb,upref}
\usepackage[autostyle]{csquotes}

\let\<\langle
\let\>\rangle

\let\uml\"

\title[Zero divisors in complex group algebra of torsion-free groups]{A necessary condition for zero divisors in complex group algebra of torsion-free groups} 

\author[A. Abdollahi]{Alireza  Abdollahi}  
\address{Department of Mathematics,  University of Isfahan, Isfahan, 81746-73441, Iran} 
\email{a.abdollahi@math.ui.ac.ir}
\author[M. Soleimani Malekan]{Meisam Soleimani Malekan}  
\address{Postdoctoral researcher, National Elites Foundation, Tehran, Iran} 
\email{m.soleimani@sci.ui.ac.ir} 
\thanks{} 



\newcommand{\CF}{\mathcal F}
\newcommand{\FX}{\mathfrak X}
\newcommand{\FK}{\mathfrak K}
\newcommand{\bC}{\mathbb C}

\newcommand{\bR}{\mathbb R}

\newcommand{\se}[1]{\left\{#1\right\}}

\newcommand{\valu}[2]{\left\langle #1,#2\right\rangle}
\newcommand{\ra}[2]{#1\rightarrow #2}
\newcommand{\pr}[1]{\left(#1\right)}
\newcommand{\supp}[1]{\text{supp}\left(#1\right)}

     \newtheorem{theorem}{Theorem}[section]
     \newtheorem{lemma}{Lemma}[section]
     \newtheorem{proposition}{Proposition}[section]
     \newtheorem{corollary}{Corollary}[section]
     
     \theoremstyle{definition}
      \newtheorem{definition}{Definition}[section]

        \newtheorem{conj}{Conjecture}[section]

\begin{document} 
 
\begin{abstract}  
 It is proved that if $\sum_{g\in G} a_g g$ is a non-zero element of the complex group  algebra $\mathbb{C}G$ of a torsion-free group $G$ which is zero divisor then  $2\sum_{g\in G} |a_g|^2<\big( \sum_{g\in G} |a_g|\big)^2$. So, for example, elements in the form $3d+\sum_{k=1}^dx_k$ for a positive integer $d>2$, or $4+x+y$ are not zero divisors in the $\mathbb C$-group algebra, and hence in $\mathbb Z$-group algebra of an arbitrary torsion free group. 
\end{abstract} 
\maketitle

\section{Introduction}
Let $G$ be any group and  $\bC G$ be the complex group algebra of $G$, i.e. the set of finitely supported complex  functions on $G$. We may represent an element $\alpha$ in $\bC G$ as a formal sum $\sum_{g\in G}a_gg$, where $a_g\in\bC$ is the value of $\alpha$ in $g$. The multiplication in $\bC G$ is defined by 
\[ \alpha\beta=\sum_{g,h\in G}a_gb_hgh=\sum_{g\in G}\pr{\sum_{x\in G}a_{gx^{-1}}b_x}x\] 
for $\alpha=\sum_{g\in G}a_gg$ and $\beta=\sum_{g\in G}b_gg$ in $\bC G$. We shall say that $\alpha$ is a \textit{zero divisor} if there exists $0\neq\beta\in\bC G$ such that $\alpha\beta=0$. If there is a non-zero $\beta\in\ell^2(G)$ such that $\alpha\beta=0$, then we may say that $\alpha$ is \textit{analytical zero divisor}. If $\alpha\beta\neq0$ for all $0\neq\beta\in\bC G$, then we say that $\alpha$ is \textit{regular}. The following conjecture is called the \textit{zero divisor conjecture}.
\begin{conj}\label{ZDC}
Let $G$ be a torsion-free group. Then all elements in $\bC G$ are regular. 
\end{conj} 
Amazingly, this conjecture has held up for many years. The conjecture \ref{ZDC} has been proven affirmative when $G$ belongs to special classes of groups; ordered groups (\cite{malcev} and \cite{neumann}), supersolvable groups (\cite{formanek}), polycyclic-by-finite groups (\cite{brown} and \cite{farsni}) and uniqe product groups (\cite{cohen-zd}). Delzant \cite{140} deals with group rings of word-hyperbolic groups and proves the conjecture for certain word-hyperbolic groups. Let $\mathcal C$ (Linnell's class of groups) be the smallest class of groups which contains all free groups and is closed under directed unions and extensions with elementary amenable quotients. Let $G$ be a group in $\mathcal C$ such that there is an upper bound on the orders of finite subgroups, then $G$ satisfies the above conjecture (\cite{lin309}).

The map $\valu\cdot\cdot: \bC G \times \bC G \rightarrow \mathbb{R}$ defined by  
\begin{align*}
\valu\alpha\beta:=\sum_{g\in G}a_g\bar{b}_g\quad(\alpha,\beta\in\bC G)
\end{align*}
is an inner product on $\bC G$, so $\|\alpha\|_2=\valu\alpha\alpha^\frac12$ becomes a norm, called 2-norm; the completion of $\bC G$ w.r.t. 2-norm is the Hilbert space $\ell^2(G)$. Indeed, we have 
\[ \ell^2(G)=\se{\alpha:\ra G\bC:\sum_{g\in G}\|\alpha(g)\|^2<\infty}.\]
In \cite{lin}, Linnell formulated an analytic version of the zero divisor conjecture.
\begin{conj}\label{AZDC}
Let $G$ be a torsion-free group. If $0\neq\alpha\in\bC G$ and $0\neq\beta\in\ell^2(G)$,
then $\alpha\beta\neq0$.
\end{conj}
In \cite{lingva}, it is shown that 
Since $\bC G\subset \ell^2(G)$, the second conjecture implies the first one. In \cite{elek}, it is proved that for finitely generated amenable groups, the two conjectures
are actually equivalent. We prove this is true for all amenable torsion-free groups. \par 
The so-called 1-norm is defined on $\bC G$ by 
\[ \|\alpha\|_1=\sum_{g\in G}|a_g|,\quad\text{for $\alpha=\sum_{g\in G}a_gg$ in $\bC G$}.\]

The \textit{adjoint} of an element $\alpha=\sum_{g\in G}a_gg$ in $\bC G$, denoted by $\alpha^*$, is $\alpha^*=\sum_{g\in G}\bar a_gg^{-1}$. We call an element $\alpha\in\bC G$ \textit{self-adjoint} if $\alpha^*=\alpha$, and use $(\bC G)_{\textbf{s}}$ to denote the set of self-adjoint elements of $\bC G$. It is worthy of mention that if $\alpha=\sum_{g\in G}a_gg$ is self-adjoint then $a_1$ should be a real number. 
For $\alpha\in\bC G$, $\beta$ and $\gamma$ in $\ell^2(G)$, the following equalities hold: 
\[ \valu{\alpha\beta}\gamma=\valu{\beta}{\alpha^*\gamma}.\]
\par 
 The goal of this paper is to give a criterion for an element in a complex group algebra to be regular:
\begin{theorem}
\label{Main} Let $G$ be a torsion free group. Then $\alpha\in\bC G$ is regular if $2\|\alpha\|_2^2\geq\|\alpha\|_1^2$.
\end{theorem}

\section{Preliminaries}
In this section we provide some preliminaries needed in the following. 
\par Let $G$ be a group. The \textit{support} of an element $\alpha=\sum_{g\in G}a_gg$ in $\bC G$, $\supp\alpha$, is the finite subset $\se{g\in G: a_g\neq0}$ of $G$.\par  Let $H$ be a subgroup of $G$, and $T$ be a right transversal for $H$ in $G$. Then every element 
$\alpha\in\bC G$ (resp. $\alpha\in\ell^2(G)$) can be written uniquely as a finite sum of the form $\sum_{t\in T}\alpha_tt$ with $\alpha_t\in\bC H$ (resp. $\alpha_t\in\ell^2(H)$). \par 
 For $S\subset G$, we denote by $\langle S\rangle$, the subgroup of $G$ generated by $S$. We have the following key lemma: 
\begin{lemma}\label{iff} 
Let $G$ be a group, $\alpha\in\bC G$ and $H=\langle\supp\alpha\rangle$. Then $\alpha$ is regular in $\bC G$ iff $\alpha$ is regular in $\bC H$.
\end{lemma}
\begin{proof}
Suppose that $\alpha$ is a zero divisor. Among elements $0\neq\gamma$ in $\bC G$ which satisfy $\alpha\gamma=0$ consider an element $\beta$ such that $1\in\supp\beta$ and $|\supp\beta|$ is minimal, then one can easily show that $\beta\in\bC H$, and this proves the result of the lemma.
\end{proof}
An immediate consequence of this lemma is:
\begin{corollary}
A group $G$ satisfies the Conjecture \ref{ZDC} iff all its finitely generated subgroups satisfy the Conjecture \ref{ZDC}. 
\end{corollary}
By Lemma \ref{iff} in hand, we can generalize the main theorem of \cite{elek}:
\begin{theorem}\label{generalize}
Let $G$ be an amenable group. If $0\neq\alpha\in\bC G$, $0\neq\beta\in\ell^2(G)$ and $\alpha\beta=0$, then there exists $0\neq\gamma\in\bC G$ such that $\alpha\gamma=0$.
\end{theorem}
 
 The above theorem along with results in \cite{KLM} provides another proof for \cite[Theorem 2]{lingva}.\par 
 For a normal subgroup $N$ of a group $G$, we denote the natural quotient map by $q_N:\ra G{G/N}$. We continue to show that:
\begin{lemma}\label{11+g}
Let $N$ be a normal subgroup of a group $G$ satisfying Conjecture \ref{ZDC}. Consider a non-torsion element $q_N(t)$, $t\in G$, in the quotient group. Then $\alpha+\beta t$ is regular, for all $\alpha,\beta\in\bC N\setminus\se{0}$.
\end{lemma}
\begin{proof}
Suppose that $\alpha+\beta t$ is a zero divisor for non zero elements $\alpha,\beta\in\bC N$. Applying Lemma \ref{iff} and multiplying by a suitable power of $t$, we can assume that there are non zero elements $\gamma_k$, $k=0,1,\dotsc, n$, such that 
\[ (\alpha+\beta t)\sum_{k=0}^n\gamma_kt^k=0.\]
In particular, $0=\beta t\gamma t^n=(\beta t\gamma_nt^{-1})t^{n+1}$, whence $\beta t\gamma_nt^{-1}=0$, a contradiction, because $t\gamma_n t^{-1}$ is a non zero element of $\bC N$.  
\end{proof}

\begin{proposition}\label{1+g}
Let $N$ be an amenable normal subgroup of a group $G$ satisfying Conjecture \ref{ZDC}. Consider a non-torsion element $q_N(t)$, $t\in G$, in the quotient group. Then there is no $0\neq\gamma\in\ell^2(G)$ such that $(\alpha+\beta t)\gamma=0$. In particular, $a+bg$ is an analytical zero divisor, for all non-torsion element $g\in G$ and non zero complex numbers $a,b$.
\end{proposition} 
\begin{proof}
The group $\valu Nt$ is amenable. Hence Lemma \ref{11+g} together with Theorem \ref{generalize} yields the result.
\end{proof}

\section{A cone of regular elements}
The result of the Proposition \ref{1+g} is true if we replace $\bC$ by an arbitrary field $\mathbb F$. The field of complex numbers allows us to define inner product on the group algebra; with the help of inner product, we can construct new regular elements from the ones we have: 
\begin{proposition}
\label{sumofnonzerdivisor}
Let $G$ be a group and $\CF$ be a finite non-empty subset of $\bC G$. If $\sum_{\alpha\in\CF}\alpha^*\alpha$ is an analytical zero divisor then all  elements of  $\CF$ are analytical zero divisors. In particular, $\alpha\in \bC G$ is an analytical zero divisor if and only if $\alpha^* \alpha$ is an analytical zero divisor. 
\end{proposition}
\begin{proof}
Let $\tilde{\alpha}:=\sum_{\alpha\in\CF}\alpha^*\alpha$ and $\tilde{\alpha}\beta=0$ for some $\beta\in\ell^2(G)$. Then 
\[ 0=\valu{\tilde\alpha\beta}\beta=\sum_{\alpha\in\CF}\valu{\alpha^*\alpha\beta}\beta=
\sum_{\alpha\in\CF}\valu{\alpha\beta}{\alpha\beta}=\sum_{\alpha\in\CF}\|\alpha\beta\|_2^2,\]
whence $\alpha\beta=0$ for all $\alpha\in\CF$. This completes the proof.
\end{proof} 
A \textit{cone} in a vector space $\FX$ is a subset $\FK$ of $\FX$ such that $\FK+\FK\subset\FK$ and $\mathbb R_+\FK\subset\FK$. We proceed by introducing a cone of regular elements in $\bC G$. First a definition: 
\begin{definition}\label{injpos} Let $G$ be a group and $(\bC G)_{\text{s}}$ be the set of self adjoint elements $\alpha\in\bC G$, we define a function $\Upsilon:\ra{(\bC G)_{\text{s}}}{\bR}$ by 
\[ \Upsilon(\alpha):=a_1-\sum_{g\neq 1}|a_g|. \]
We call an element $\alpha\in(\bC G)_{\text{s}}$ \textbf{golden} if $\Upsilon(\alpha)\geq0$. The set of all golden elements in $(\bC G)_{\text{s}}$ is denoted by $(\bC G)_{\text{gold}}$. 
\end{definition}
What is important about golden elements is:
\begin{proposition}\label{cone}
For a torsion free group $G$, $(\bC G)_{\text{gold}}$ is a cone of regular elements.
\end{proposition}
\begin{proof}
It is obvious that if $\alpha$ is golden then so is $r\alpha$ for any $r>0$. The triangle inequality for $\bC$ shows that if $\alpha$ and $\gamma$ are golden then so is $\alpha+\gamma$. For $\alpha\in(\bC G)_{\text{s}}$, we have 
\begin{align*}\label{sumpos}
\alpha&=\frac12(\alpha+\alpha^*)\notag\\&= a_1+\frac12\sum_{g\neq1}\pr{\bar a_gg^{-1}+a_gg}\notag\\
&=\Upsilon(\alpha)+\frac12\sum_{g\neq1}\pr{2|a_g|+\bar a_gg^{-1}+a_gg}\\
&=\Upsilon(\alpha)+\frac12\sum_{g\neq1}|a_g|\pr{\frac{\bar a_g}{|a_g|}+g}^*\pr{\frac{a_g}{|a_g|}+g}
\end{align*}
Hence, by Lemma \ref{1+g} and Proposition \ref{sumofnonzerdivisor}, $\alpha$ is regular. 
\end{proof} 
Now, we are ready to prove our main result:
\begin{proof}[Proof of Theorem \ref{Main}]
For $\alpha=\sum_{g\in G}a_gg$ in $\bC G$, $\alpha^*\alpha$ is self-adjoint, and one can easily show that 
\[\Upsilon(\alpha^*\alpha)\geq2\|\alpha\|_2^2-\|\alpha\|_1^2.\] Hence, by Proposition \ref{cone}, the result of the Theorem is proved. 
\end{proof}

\bibliography{Bib}{}

\begin{thebibliography}{10}

\bibitem{brown}
K.~A. Brown.
\newblock On zero divisors in group rings.
\newblock {\em Bull. Lond. Math. Soc.}, 8:251--256, 1976.

\bibitem{cohen-zd}
J.~M. Cohen.
\newblock Zero divisors in group rings.
\newblock {\em Comm. Algebra}, 2:1--14, 1974.

\bibitem{140}
T.~Delzant.
\newblock Sur l’anneau d’un groupe hyperbolique.
\newblock {\em C. R. Acad. Sci. Paris S’er. I Math.}, 324(4):381--384, 1997.

\bibitem{elek}
Gabor Elek.
\newblock On the analytic zero divisor conjecture of {L}innell.
\newblock {\em Bull. Lond. Math. Soc.}, 35(2):236--238, 2003.

\bibitem{farsni}
D.~R. Farkas and R.~L. Snider.
\newblock $k_0$ and {N}oetherian group rings.
\newblock {\em J. Algebra}, 42:192--198, 1976.

\bibitem{formanek}
E.~Formanek.
\newblock The zero divisor question for supersolvable groups.
\newblock {\em Bull. Aust. Math. Soc.}, 73c:67--71, 1973.

\bibitem{lingva}
P.~A. Linnell.
\newblock Zero divisors and group von {N}eumann algebras.
\newblock {\em Pacific J. Math.}, 149:349--363, 1991.

\bibitem{lin309}
P.~A. Linnell.
\newblock Division rings and group von neumann algebras.
\newblock {\em Forum Math.}, 5(6):561--576, 1993.

\bibitem{lin}
P.~A. Linnell.
\newblock Analytic versions of the zero divisor conjecture.
\newblock In {\em Geometry and cohomology in group theory (Durham 1994)}, 252.
  London Math. Soc. Lecture Note, 1998.

\bibitem{malcev}
A.~I. Malcev.
\newblock On embedding of group algebras in a division algebra (in {R}ussian).
\newblock {\em Dokl. Akad. Nauk}, (60):1499--1501, 1948.

\bibitem{neumann}
B.~H. Neumann.
\newblock On ordered division rings.
\newblock {\em Trans. Amer. Math. Soc.}, 66:202--252, 1949.

\bibitem{KLM}
P.~A.~Linnell P.~H.~Kropholler and J.~A. Moody.
\newblock Applications of a new {$K$}-theoretic theorem to soluble group rings.
\newblock {\em Proc. Amer. Math. Soc.}, 104:675--684, 1988.

\end{thebibliography}
\bibliographystyle{plain}
\end{document}